\documentclass[11pt]{scrartcl}
\usepackage[utf8]{inputenc}
\usepackage[T1]{fontenc}
\usepackage{color}
\usepackage{enumerate}
\usepackage{titling}
\usepackage{relsize,xspace}
\usepackage[usenames,dvipsnames]{xcolor}

\usepackage{microtype}
\usepackage{amsmath}
\usepackage{amssymb}
\usepackage{amsfonts}
\usepackage{stmaryrd}
\usepackage{tikz}
\usepackage{bm}
\usepackage{algorithm}
\usepackage{algpseudocode}
\usepackage{todonotes}
\usepackage{dsfont}
\usepackage{nicefrac}
\usepackage[inline]{enumitem}
\setlist[itemize]{topsep=0pt,partopsep=0pt,itemsep=0pt,parsep=0pt}
\setlist[itemize,1]{label=\textbullet}
\setlist[itemize,2]{label=---}
\setlist[itemize,3]{label=*}
\setlist[enumerate]{topsep=0pt,partopsep=0pt,itemsep=0pt,parsep=0pt}
\setlist[enumerate,1]{label=\textit{(\roman*)}}
\setlist[enumerate,2]{label=(\alph*)}
\setlist[enumerate,3]{label=(\arabic*)}

\usetikzlibrary{calc,through,arrows,arrows.meta,automata,decorations.pathmorphing,decorations.pathreplacing,shapes,backgrounds,fit,positioning,shapes.geometric,patterns}
\usetikzlibrary{decorations.text}
\usetikzlibrary{decorations}
\tikzset{
	position/.style args={#1:#2 from #3}{
		at=($(#3)+(#1:#2)$)
	},
	v:main/.style = {draw, circle, scale=0.6, thick},
	v:matching/.style = {draw, circle, scale=0.6, thick,fill=matEdgeColour},
	e:undir/.style = {draw,line width=1pt},
	e:dir/.style = {draw,->,line width=1.3pt}
}

\pgfdeclarelayer{bg}    % declare background layer
\pgfsetlayers{bg,main}  % set the order of the layers (main is the standard layer)

\definecolor{blackblue}{rgb}{0,0.18,0.39}
\usepackage[ocgcolorlinks, allcolors={blackblue}]{hyperref}

\definecolor{magenta}{rgb}{0.79, 0.08, 0.48}
\definecolor{AO}{rgb}{0.0, 0.5, 0.0}
\definecolor{phthaloblue}{rgb}{0.0, 0.06, 0.54}
\definecolor{pistachio}{rgb}{0.58, 0.77, 0.45}
\definecolor{darkgoldenrod}{rgb}{0.72, 0.53, 0.04}
\definecolor{matEdgeColour}{rgb}{0.52, 0.73, 0.4}
\colorlet{matEdgeColourText}{matEdgeColour!60!black}
\definecolor{myBlue}{rgb}{0.25, 0.0, 1.0}
\definecolor{myRed}{rgb}{0.7, 0.11, 0.11}
\definecolor{myGreen}{rgb}{0.13, 0.55, 0.13}
\definecolor{typecol}{rgb}{0.13, 0.55, 0.13}
\colorlet{myViolet}{myBlue!55!myRed}

\usepackage[amsmath,thmmarks,hyperref]{ntheorem}
\usepackage{cleveref}
\usepackage{macros}

\usepackage[utf8]{inputenc}

\usepackage{latexsym}

\title{Cyclewidth and the Grid Theorem for\\ Perfect Matching Width\\ of Bipartite Graphs\thanks{This work has been supported by the European Research Council (ERC) under the European Union’s Horizon 2020 research and innovation programme (ERC consolidator grant DISTRUCT, agreement No 648527).}}
\predate{}
\date{}
\postdate{}

\preauthor{}
\DeclareRobustCommand{\authorthing}{
	\begin{center}
		\begin{tabular}{p{0.18\textwidth}p{.25\textwidth}p{.25\textwidth}}
                  Meike Hatzel & Roman Rabinovich & Sebastian Wiederrecht\\
		\end{tabular}
		
		\{meike.hatzel, roman.rabinovich, sebastian.wiederrecht\}@tu-berlin.de\\
		Technische Universität Berlin
\end{center}}
\author{\authorthing}
\postauthor{}
\begin{document}
	
	\maketitle
	
	\setcounter{page}{1}
	
	\begin{abstract}
          A connected graph $G$ is called \emph{matching covered} if
          every edge of $G$ is contained in a perfect matching.
          \emph{Perfect matching width} is a width parameter for
          matching covered graphs based on a branch decomposition.  It
          was introduced by Norine and intended as a tool for the
          structural study of matching covered graphs, especially in
          the context of Pfaffian orientations.  Norine conjectured
          that graphs of high perfect matching width would contain a
          large grid as a matching minor, similar to the result on
          \treewidth by Robertson and Seymour.
		
          In this paper we obtain the first results on perfect
          matching width since its introduction.  For the restricted
          case of bipartite graphs, we show that perfect matching
          width is equivalent to directed \treewidth and thus the
          Directed Grid Theorem by Kawarabayashi and Kreutzer for
          directed \treewidth implies Norine's conjecture.
		
	\end{abstract}

\section{Introduction}
	
	The concept of width-parameters, or decompositions of graphs into tree like structures has proven to be a powerful tool in both structural graph theory and for coping with computational intractability.
	The shining star among these concepts is the \emph{\treewidth} of undirected graphs introduced in its popular form in the Graph Minor series by Robertson and Seymour (see \cite{graphminorproject}).
	
	Tree decompositions are a way to decompose a given graph into
	loosely connected small subgraphs of bounded size that, in many algorithmic applications, can be dealt with individually instead of considering the graph as a whole.
	This concept allows the use of dynamic programming and other
	techniques to solve many hard computational problems, see for example~\cite{bodlaender1996linear,bodlaender1997treewidth,bodlaender2005discovering,downey2016fundamentals}.
	
	\Treewidth was also successfully applied for non-algorithmic problems:
	in the famous Graph Minor project by Robertson and Seymour \treewidth plays a key role, in model
	theory~\cite{Gr99d,BaranytCatSeg11,SegoufinCate13,benediktBouBoo2017},
	or in the proofs of the (general) Erd\H{o}s-P\'osa property for	undirected graphs~\cite{robertson1986graph}.
	
	In the latter result, the \emph{Grid theorem} plays an important role.
	It says	that if the \treewidth of a graph is big, then it has a large
	grid as a minor.
	Proven by Robertson and Seymour in 1986~\cite{robertson1986graph}, the
	Grid Theorem has given rise to many other interesting results.
	An example of those is the algorithm design principle called
	\emph{bidimensionality theory,} that, roughly, consists in
	distinguishing two cases for the given graph: small \treewidth or large grid minor (see \cite{demaine2007bidimensionality,demaine2004fast,fomin2010bidimensionality} for examples).
	
	As directed graphs pose a natural generalisation of graphs, soon the question arose whether a similar strategy would be useful for directed graphs and so Reed~\cite{reed1999directed} and Johnson et al.~\cite{johnson2001directed} introduced \emph{\dtwText} along with the conjecture of a directed version of the Grid Theorem.
	After being open for several years, the conjecture was finally proven
	by Kawarabayashi and Kreutzer~\cite{kawarabayashi2015directed}.
	Similarly to the undirected case, it implies the Erd\H{o}s-P\'{o}sa property for
	large classes of directed graphs~\cite{amiri2016erdos}.
	
	It is possible to go further and to consider even more general structures than directed graphs.
	One of the ways to do this is to characterise (strongly connected) directed graphs by pairs of undirected bipartite graphs and \emph{perfect matchings}.
	The generalisation (up to strong connectivity) is then to drop the condition on the graphs to be bipartite.
	There is a deep connection between the theory of \emph{matching minors} in matching covered graphs and the theory of \emph{butterfly minors} and strongly connected directed graphs.
	This connection can be used to show structural results on directed graphs by using matching theory (see \cite{mccuaig2000even,guenin2011packing}).
	
	The corresponding branch of graph theory was developed from the theory of
	\emph{tight cuts} and \emph{tight cut decompositions} of
	\emph{matching covered} graphs introduced by Kotzig, Lov{\'a}sz and
	Plummer~\cite{lovasz1987matching,lovasz2009matching,kotzig1959theory}.
	A graph is matching covered if it is connected and each of its edges is contained in a perfect matching.
	One of the main forces behind the development of the field is the question of Pfaffian orientations; see~\cite{mccuaig2004polya,thomas2006survey} for an overview on the subject.
	
	Matching minors can be used to characterise the bipartite Pfaffian graphs~\cite{mccuaig2004polya,robertson1999perma}. The characterisation  implies a polynomial time algorithm for the problem to decide whether a matching covered bipartite graph is Pfaffian or not.
	In addition, there are powerful generation methods for the building blocks of matching covered graphs obtained by the tight cut decomposition: the \emph{bricks} and \emph{braces}, which are an analogue of Tutte's theorem on the generation of $3$-connected graphs from wheels.
	Here the bipartite (brace) case was solved by McGuaig~\cite{mccuaig2001brace} while the non-bipartite (brick) case was solved by Norine in his PhD thesis~\cite{norine2005matching}.
	The goal of this thesis was to find an analogue to the bipartite characterisation of Pfaffian graphs in the non-bipartite case.
	However, while a bipartite matching covered graph turns out to be Pfaffian if and only if it does not contain $K_{3,3}$ as a matching minor, Norine discovered an infinite antichain of non-Pfaffian bricks.
	
	While no polynomial time algorithm for Pfaffian graphs is known,
	Norine defined a branch decomposition for matching covered graphs and
	found an algorithm that decides whether a graph from a class of bounded \emph{perfect matching width} is Pfaffian in \XP-time.
	This branch decomposition for matching covered graphs and perfect matching width are similar to \emph{branch decompositions} and \emph{branchwidth}, again introduced by Robertson and Seymour~\cite{robertson1991graph}.
	Norine and Thomas also conjectured a grid theorem for their new width parameter (see \cite{norine2005matching,thomas2006survey}).
	Based on the above mentioned ties between bipartite matching covered graphs and
	directed graphs, Norine conjectures in his thesis that the 
	Grid Theorem for digraphs, which was still open at that time, would at least imply the conjecture in the bipartite case.
	Whether perfect matching width and \dtwText could be seen as equivalent was unknown at that time.
	
	\paragraph*{Contribution.}
	
	We settle the \emph{Matching Grid Conjecture} for the bipartite case.
	To do so, in \cref{sec:dtw_cycw}, we construct a branch decomposition and a corresponding new width parameter for directed graphs: the \emph{\cyclewidth} and prove its equivalence to \dtwText.
	
	\Cyclewidth itself seems to be an interesting parameter as it appears to be more natural and gives further insight in the difference between undirected and \dtwText: while the undirected case considers local properties of the graph, the directed version is forced to have a more global point of view.
	We also prove that \cyclewidth is closed under butterfly minors, which
	is not true for \dtwText as shown by Adlern~\cite{adler2007directed}.
	
	The introduction of \cyclewidth leads to a straightforward proof of
	the Matching Grid Theorem for bipartite graphs. In \cref{sec:pmw} we show that \cyclewidth and perfect matching width are within a constant factor of each other. This immediately implies the Matching Grid Theorem for bipartite graphs. Our proofs are algorithmic and thus also imply an approximation algorithm for perfect matching width on bipartite graphs, which is the first known result on this matter.
	
	Norine proposes the quadratic
	planar grid as the right matching minor witnessing high perfect
	matching width as in the original Grid Theorem by Robertson and
	Seymour. In this work we give an argument that \emph{cylindrical grids} are a more natural grid-like structure in the context of matching covered
	graphs. Moreover, to better fit into the canonical matching theory, we
	would like our grid to be a brace. McGuaig proved that every brace
	either contains $K_{3,3}$, or the \emph{cube} as a matching minor (see
	\cite{mccuaig2001brace}). Additionally, one of the three infinite
	families of braces from which every brace can be generated are the
	even prisms (or planar ladders as McGuaig calls them in his paper) of
	which the cube is the smallest one.
	The \emph{bipartite matching grid,} which we define in this work can be seen
	as a generalisation of the even prism and again the cube is the
	smallest among them.

	Width parameters for directed graphs similar to directed treewidth are exclusively concerned with directed cycles.
	So, when studying \cyclewidth and related topics one might restrict themselves to strongly connected digraphs.
	Similarly, an edge that is not contained in any perfect matching is, in most cases, irrelevant for the matching theoretic properties of the graph.
	For this reason it is common to only consider matching covered graphs as this does not pose a loss of generality.
	
	Finally, we show that the perfect matching widths of a
        matching covered bipartite graph and of any matching minor of
        the graph are within a constant factor of each other.
	
	Let us remark that we took the freedom to rename Norine's  \emph{matching-width}~\cite{norine2005matching} to \emph{perfect matching width} to better distinguish it from related parameters such as \emph{maximum matching width} (see \cite{jeong2015maximum}).

\section{Preliminaries}
        
	We consider finite graphs and digraphs without multiple edges and use
	standard notation (see~\cite{diestel2017graph}). For a
	graph $G$, its vertex set is denoted by $V(G)$ and its edge set by
	$E(G)$, and similarly for digraphs where we call arcs
        edges. For a (directed) tree $T$, we write $L(T)$ for the set
        of its leaves. 
	
	Let $X \subseteq V(G)$ be a non-empty set of vertices in a graph $G$. 
	The \emph{cut at $X$} is the set $\cut{X}\subseteq\E{G}$ of all edges
	joining vertices of $X$ to vertices of $\V{G}\setminus X$. 
	We call $X$ and $\V{G}\setminus X$ the \emph{shores} of
	$\cut{\Shore}$. A set $E\subseteq E(G)$ is a \emph{cut} if $E =
	\cut{X}$ for some $X$. Note that in connected graphs the shores are
	uniquely defined. In such cases, a cut is said to be \emph{odd} if both shores have odd cardinality and we call a cut \emph{trivial} if one of the two shores only contains one vertex.
	
	A \emph{matching} of a graph $G$ is a set $\Matching\subseteq\E{G}$ such that no two edges in $\Matching$ share a common endpoint.
	If $e=xy\in \Matching$, $e$ is said to \emph{cover} the two vertices $x$ and $y$.
	A matching $\Matching$ is called \emph{perfect} if every vertex of $G$ is covered by an edge of $\Matching$.
	We denote by $\perf{G}$ the set of all perfect matchings of a graph $G$.
	A restriction of a matching $M$ to a set $S\subseteq\V{G}$ or to a
	subgraph $G'\subseteq G$ is defined by $\reduct{\Matching}{S} \coloneqq \condset{xy \in \Matching}{x,y \in S}$ and $\reduct{\Matching}{G'} \coloneqq \reduct{\Matching}{\V{G'}}$.
	
	\begin{definition}
		\label{def:Mdirection}
		Let $G=\br{A\cup B, E}$ be a bipartite graph and let $\Matching\in\perf{G}$ be a perfect matching of $G$. 
		The \emph{$\Matching$-direction} $\dirm{G}{\Matching}$ of $G$ is defined as follows (see \cref{fig:exMdirection} for an illustration).
		Let $\Matching = \Set{a_1b_1,\dots,a_{\Abs{\Matching}}b_{\Abs{\Matching}}}$
		with $a_i\in A, b_i\in B$ for $1\le i\le \Abs{\Matching}$.
		Then,
		\begin{enumerate}
			\item $\V{\dirm{G}{\Matching}}\define\Set{v_1,\dots,v_{\Abs{\Matching}}}$ and
			
			\item $\E{\dirm{G}{\Matching}}\define\condset{\br{v_i,v_j}}{a_ib_j\in\E{G}}$.	
		\end{enumerate}
	\end{definition}
	
	\begin{figure}[ht]
		\begin{center}
			\begin{tikzpicture}
			
			\node (centre) [] {};
			
			\node (anchorG) [position=180:3cm from centre] {};
			
			\node(aGleft) [position=180:1cm from anchorG] {};
			
			\node(labelG) [position=139:23mm from aGleft] {$G$ and $\textcolor{matEdgeColourText}{\Matching}$};
			
			\foreach\i in {1,...,4} {
				%w
				\node [v:main, position=67.5+\i*90:12mm from aGleft] (w-\i) {};	
				\node [scale=1, position=\i*90+90:4mm from w-\i] (w-\i-Label) {$b_{\i}$};
				%u
				\node [v:main, position=112.5+\i*90:12mm from aGleft] (u-\i) {};
				\node [scale=1, position=\i*90+90:4mm from u-\i] (u-\i-Label) {$a_{\i}$};
			}
			
			\node(aGright) [position=0:12mm from anchorG] {};
			
			\foreach\i in {5,...,7} { 
				%w
				\node [draw, circle, scale=0.6, thick, position=202.5+\i*90:12mm from aGright] (w-\i) {};	
				\node [scale=1, position=\i*90+180:4mm from w-\i] (w-\i-Label) {$b_{\i}$};
				%u
				\node [draw, circle, scale=0.6, thick, position=157.5+\i*90:12mm from aGright] (u-\i) {};
				\node [scale=1, position=\i*90+180:4mm from u-\i] (u-\i-Label) {$a_{\i}$};
			}
			
			\node [scale=0.6, thick, position=202.5+8*90:12mm from aGright] (w-8) {};	
			\node [scale=0.6, thick, position=157.5+8*90:12mm from aGright] (u-8) {};
			
			\node (anchorD) [position=0:4cm from centre] {};
			
			\node(aDleft) [position=180:1cm from anchorD] {};
			
			\node(labelD) [position=139:23mm from aDleft] 
			{$\dirm{G}{{\textcolor{matEdgeColourText}{\Matching}}}$};		
			
			\foreach\i in {1,...,4} {
				%v
				\node [v:matching, position=90+\i*90:9mm from aDleft] 
				(v-\i) {};	
				\node [scale=1, position=\i*90+90:4mm from v-\i] (v-\i-Label) {$v_{\i}$};	
			}
			
			\node(aDright) [position=0:8mm from anchorD] {};
			
			\foreach\i in {5,...,7} { 
				%v
				\node [v:matching, position=180+\i*90:9mm from aDright] 
				(v-\i) {};	
				\node [scale=1, position=\i*90+180:4mm from v-\i] (v-\i-Label) {$v_{\i}$};		
			}
			
			\node [scale=0.6, thick, position=180+8*90:9mm from aDright] (v-8) {};

			\begin{pgfonlayer}{bg}

			\foreach\i in {1,...,7} {
				%matching
				\drawMatEdge{u-\i}{w-\i};
			}
			
			%black edges left
			\foreach\i in {1,...,4} {
				\pgfmathtruncatemacro{\iNext}{Mod(\i,4)+1};
				\path (u-\i) [e:undir] edge (w-\iNext);	
				\pgfmathtruncatemacro{\iAct}{\i+4};
				\pgfmathtruncatemacro{\iNextA}{Mod(\i,4)+5};
				\path (w-\iAct) [e:undir] edge (u-\iNextA);			
			}

			%black edges right
			\foreach\i in {1,...,4} {
				\pgfmathtruncatemacro{\iNext}{Mod(\i,4)+1};
				\path[>=latex] (v-\i) edge[e:dir] (v-\iNext);
				\pgfmathtruncatemacro{\iAct}{\i+4};
				\pgfmathtruncatemacro{\iNextA}{Mod(\i,4)+5};
				\path[>=latex] (v-\iNextA) edge[e:dir] (v-\iAct);		
			}		
			
			\end{pgfonlayer}

			\end{tikzpicture}
			
		\end{center}
		\caption{A bipartite graph $G=\br{A\cup B,E}$ with perfect matching 
			${\textcolor{matEdgeColourText}{\Matching}}$ and its ${\textcolor{matEdgeColourText}{\Matching}}$-direction.}
		\label{fig:exMdirection}
	\end{figure}
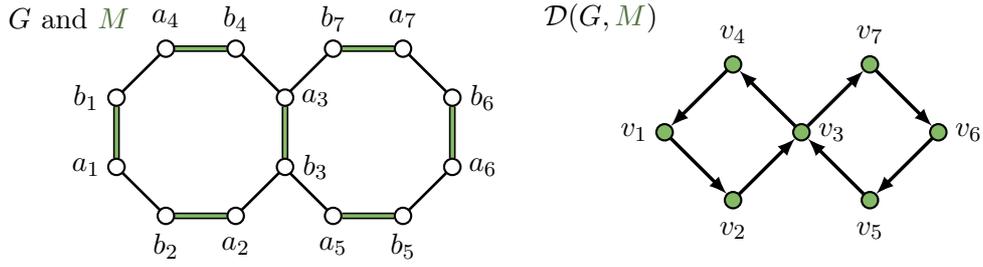
	
	Thus, the \emph{$\Matching$-direction} $\dirm{G}{\Matching}$ of $G$ is
	defined by contracting the edges of $\Matching$, and orienting the
	edges of $G$ that do not belong to $\Matching$ from $A$ to $B$. 
	
	A graph $G$ is called \emph{matching covered} if $G$ is connected and for every edge $e \in \E{G}$ there is an $\Matching \in \perf{G}$ with $e \in \Matching$. 
	
	% A digraph $D$ is called \emph{strongly connected} if for every pair of vertices $x,y\in\V{D}$ there is a directed path from $x$ to $y$ and a directed path from $y$ to $x$ in $D$.
	The following is a well known observation on $M$-directions.
	
	\begin{observation}\label{obs:strongmdirections}
		A digraph $D$ is strongly connected if and only if there is an, up to
		isomorphism unique, pair consisting of a bipartite matching covered graph $G$ and a perfect matching $\Matching\in\perf{G}$ such that $D$ is isomorphic to $\dirm{G}{\Matching}$.
	\end{observation}
	
	A set $S \subseteq \V{G}$ of vertices is called \emph{\conformal} if $G-S$ has a perfect matching.
	Given a matching $\Matching\in\perf{G}$, a set $S\subseteq \V{G}$ is called \emph{$\Matching$-\conformal} if $\reduct{\Matching}{G-S}$ is a perfect matching of $G-S$ and
	$\reduct{\Matching}{S}$ is a perfect matching of $\induc{G}{S}$.
	A subgraph $H\subseteq G$ is \emph{\conformal}  if $\V{H}$ is a \conformal set.
	$H$ is called \emph{$\Matching$-conformal} for a perfect matching $\Matching\in\perf{G}$  if $H$ is conformal and $\reduct{\Matching}{H}$ is a perfect matching of $H$.
	
	If a cycle $C$ is $\Matching$-\conformal, there is another perfect matching $\Matching'\neq \Matching$ with $\E{C} \setminus \Matching \subseteq \Matching'$.
	Hence, if needed, we say $C$ is \emph{$\Matching$-$\Matching'$-\conformal} to indicate that $\Matching$ and $\Matching'$ form a partition of the edges of $C$.
	
	\section{Directed \Treewidth and \Cyclewidth}
	\label{sec:dtw_cycw}
	
	Directed \treewidth is, similar to \treewidth on undirected graphs, an important tool in the structure theory of directed graphs.
	For \treewidth there is an equivalent (up to a constant factor) concept of \emph{branch-width} that also allows a structural comparison of an undirected graph to a tree.
	To the best of our knowledge, so far there is no branch decomposition like concept for digraphs that can be seen as an equivalent to directed \treewidth.
	For the proof of our main result, we introduce such a decomposition, which we call \emph{\cyclewidth} and prove some basic properties of it such as its equivalence to directed \treewidth and that it is closed under butterfly minors.
	
	This section is divided into two subsections.
	First we introduce \cyclewidth and show that it provides a lower bound on the directed \treewidth with a linear function.
	Then, in a second step, we show that \cyclewidth is bounded from below by the \cyclewidth of its butterfly minors and, moreover, that large cylindrical grids have large \cyclewidth.
	The Directed Grid Theorem implies that there exists a function that bounds the \cyclewidth of a digraph from below by its directed \treewidth.
	
	\subsection{Cyclewidth: A Branch Decomposition for \Digraphs}
	
	We first recall the directed \treeDecomp by Reed~\cite{reed1999directed}, and Johnson et al.~\cite{johnson2001directed}.
	
	An \emph{arborescence} is a directed tree $T$ with a root $r_0$ and all edges directed away from $r_0$.
	For $r,r'\in\V{T}$ we say that $r'$ is \emph{below} $r$ and $r$ is \emph{above} $r'$ in $T$ and write $r'>r$ if $r'\neq r$ and $r'$ is reachable from $r$ in $T$.
	For $e\in\E{T}$ with head $r$ we write $r'>e$ if either $r'=r$, or $r'>r$.
	
	\begin{definition}\label{def:splitTree}
		Let $T$ be an arborescence or a rooted undirected tree and let $e \in \E{T}$.
		Then $\splitTree{T}{e} \coloneqq \br{T_1,T_2}$ where $T_1$ is the sub-arborescence (subtree) of $T-e$ containing the root of $T$ and $T_2$ is the other sub-arborescence (subtree) with the head of $e$ as the new root.
		By a slight abuse of notation an undirected tree $T$ (without a root) and an edge $e = tt'$ in it, we write $\splitTree{T}{e} \coloneqq \br{T_1,T_2}$ where $T_1$ is the subtree containing $t$ and $T_2$ the subtree containing $t'$ in $T-e$.
	\end{definition}
	
	Let $D$ be a \digraph and let $Z\subseteq\V{D}$.
	A set $S\subseteq\V{D}-Z$ is \emph{$Z$-normal} if
        there is no directed walk in $D-Z$ with the first and the last vertex in $S$ that uses a vertex of $D-\br{Z\cup S}$.
	
	\begin{definition}
		A \emph{directed \treeDecomp} of a \digraph $D$ is a
                triple $\br{T,\beta,\gamma}$, where~$T$ is an
                arborescence and $\beta\colon\V{T}\rightarrow2^{V\br{D}}$ and $\gamma\colon\E{T}\rightarrow2^{V\br{D}}$ are functions such that
		\begin{enumerate}
			\item $\condset{\fkt{\beta}{t}}{t\in\V{T}}$ is a
			partition of $\V{D}$ into possibly empty sets
                        (a near partition) and
			\item if $e\in\E{T}$, then $\bigcup\condset{\fkt{\beta}{t}}{t\in\V{T},t>e}$ is $\fkt{\gamma}{e}$-normal.
		\end{enumerate}
		For any $t\in\V{T}$ we define $\fkt{\Gamma}{t}\define\fkt{\beta}{t}\cup\bigcup\condset{\fkt{\gamma}{e}}{e\in\E{T},e \sim t}$, where $e\sim t$ if $e$ is incident with $t$.
		The \emph{width} of $\br{T,\delta,\gamma}$ is
                $\max_{t\in \V{T}}\Abs{\fkt{\Gamma}{t}}-1$.
		The \emph{directed \treewidth} $\dtw{D}$ of $D$ is the least
		integer $w$ such that $D$ has a directed \treeDecomp of width $w$.
		The sets $\fkt{\beta}{t}$ are called \emph{bags} and the sets $\fkt{\gamma}{e}$ are called the \emph{guards} of the directed \treeDecomp.
	\end{definition}

	We also apply the bag-function $\beta$ on subtrees instead of single vertices to refer to the union over all bags in the subtree, i.e.\@ $\fkt{\beta}{T'} \coloneqq \bigcup_{v \in \V{T}} \fkt{\beta}{v}$ for $T'$ being a subtree $T$.
	If a vertex $v$ of $D$ is contained in $\fkt{\beta}{T'}$ for some subtree $T'\subseteq T$, we say that $T'$ \emph{contains} $v$.
	
	Directed \treewidth is a generalisation of the undirected version \treewidth.
	Similar to the undirected case one can find certain structural obstructions witnessing that a digraph has high directed \treewidth.
	An important result among these is the Directed Grid Theorem by Kawarabayashi and Kreutzer~\cite{kawarabayashi2015directed}.
	It states, roughly, that a graph has high directed \treewidth if and
	only if it contains a large cylindrical grid as a butterfly minor.
	
	In order to formally state the Directed Grid Theorem, we need some further definitions.
	
	\begin{definition}[Butterfly Minor]
		\label{def:butterfly_minors}
		Let $D$ be a \digraph.
		An edge $e=\br{u,v}\in\E{D}$ is \emph{butterfly-contractible} if $e$ is the only outgoing edge of $u$ or the only incoming edge of $v$.
		
		In this case the graph $D'$ obtained from $D$ by \emph{butterfly contracting} $e$ is the graph with vertex set $\br{\V{D}\setminus\Set{u,v}}\cup\Set{x_{u,v}}$, where $x_{u,v}$ is a new vertex and the edge set
		\begin{align*}
			\E{D}&\setminus\CondSet{e}{e \sim u \text{ or } e \sim v}\\ 
			&\cup\CondSet{\Brace{w,x_{u,v}}}{\Brace{w,u} \in \E{D} \text{ or } \Brace{w,v} \in \E{D}}\\
			& \cup\CondSet{\Brace{x_{u,v},w}}{\Brace{u,w} \in \E{D} \text{ or } \Brace{v,w} \in \E{D}}.
		\end{align*}
		
		We denote the result of butterfly contracting an edge $e$ in $D$ by $D/e$.
		
		A \digraph $D'$ is a \emph{butterfly-minor} of $D$ if it can be obtained from a subgraph of $D$ by butterfly contractions.
	\end{definition}
	
	\begin{definition}[Cylindrical Grid]
		A \emph{cylindrical grid} $\CylGrid{k}$ of order $k$ consists of $k$ concentric directed cycles and $2k$ paths
		connecting the cycles in alternating directions, see \cref{fig:cyl_grid}.
		\begin{figure}[!h]
			\centering
			\begin{tikzpicture}[scale=0.4]
			\CylindricalGrid{6}{0}
			\foreach \i in {1,3,5,7,9,11}
			{
				\draw ({360/12 * \i}:6) edge[-latex,line width=1pt,myViolet] ({360/12 * \i}:1);
			}
			\foreach \i in {2,4,6,8,10,12}
			{
				\draw ({360/12 * \i}:6) edge[latex-,line width=1pt,myRed] ({360/12 * \i}:1);
			}
			\end{tikzpicture}
			\caption{A cylindrical grid of order 6.}
			\label{fig:cyl_grid}
		\end{figure}
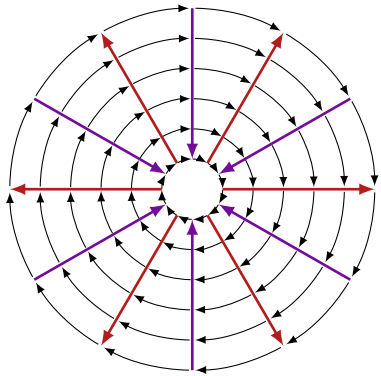
	\end{definition}
	
	Now we can state the Directed Grid Theorem.
	
	\begin{theorem}[Kawarabayashi and Kreutzer, 2015 \cite{kawarabayashi2015directed}]\label{thm:DGT}
		There is a function $f \colon\N \to \N$ such that every digraph $D$ either satisfies $\dtw{D}\leq\ f(k)$, or contains the cylindrical grid of order $k$ as a butterfly minor.
	\end{theorem}
	
	The goal of this section is the introduction of a branch decomposition for digraphs.
	Branch decompositions usually work as follows.
	They are defined as a tuple $\br{T,\delta}$ such that $T$ is a cubic tree and $\delta$ is a bijection between the leaves of $T$ and the vertices of the graph $G$ that is decomposed by $\br{T,\delta}$.
	Therefore every edge of $T$ induces a bipartition of the vertex set of $G$, which can be seen as the two shores of an edge cut.
	The width of the decomposition depends on the function that evaluates the edge cut.
	
	In the case of directed \treewidth, this concept faces a challenge.
	While edge cuts are very local objects, the guards of a directed tree decomposition are not.
	In fact, one of the main issues of directed \treewidth is that the guards can appear almost everywhere in the graph.
	
	In order to approach this problem, we need our evaluation function for the edge cuts given by our decomposition to measure a more global property.
	We define \emph{cuts}, \emph{shores} and \emph{trivial cuts}
        for \digraphs as for undirected graphs. Recall that $L(T)$ is
        the set of leaves of a tree $T$.
	
	\begin{definition}[\Cyclewidth]
		Let $D$ be a \digraph.
		A cycle decomposition of $D$ is a tuple~$\br{T,\varphi}$,
		where $T$ is a cubic tree (i.e.\@ all inner vertices have
		degree three) and $\varphi: \Leaves{T} \to \V{D}$ a bijection.
		For a subtree $T'$ of $T$ we use $\varphi(T') \coloneqq \CondSet{\fkt{\varphi}{t}}{t\in\V{T'}\cap\Leaves{T}}$.
		Let $t_1t_2$ be an edge in $T$ and let $\br{T_1,T_2} \coloneqq
		\splitTree{T}{t_1t_2}$. Let $\cut{t_1t_2} \coloneqq \cut{\phi(T_1)}$.
		The \emph{\cycPoros} of the edge~$t_1t_2$ is
		\begin{align*}
			\cycPor{\cut{t_1t_2}} \coloneqq \max_{\substack{\mathcal{C} \text{ family of pairwise}\\\text{disjoint directed cycles}\\\text{in $D$}}} \Bigl|\cut{t_1t_2} \cap \bigcup_{C\in\mathcal{C}}\E{C}\Bigr|.
		\end{align*}
		The width of a cycle decomposition $\br{T,\varphi}$ is given by $\max_{t_1t_2 \in \E{T}} \cycPor{\cut{t_1t_2}}$ and the \cyclewidth of $D$ is then defined as 
		\begin{align*}
			\cycWidth{D} \coloneqq \min_{\substack{\br{T,\varphi} \text{ cycle decomposition}\\\text{of } D}} \quad \max_{t_1t_2 \in \E{T}} \cycPor{\cut{t_1t_2}}.
		\end{align*}	
	\end{definition}

Note that the cycle porosity is the number of cycle edges crossing a
cut. Note also that for different edges of the decomposition tree,
different cycle families may constitute the cycle porosity.  

Moreover, let $D$ be a digraph and $\Brace{T,\delta}$ a cycle decomposition for $D$ of width $k$.
Let $D'$ be the digraph obtained from $D$ by reversing the orientation of all edges, then $\Brace{T,\delta}$ is a cycle decomposition for $D'$ of width $k$.
        
	Our next goal is to prove that \cyclewidth is bounded from above by a function in the directed \treewidth.
	For this, we transform a directed tree decomposition into a cycle
	decomposition in two steps. First, we push all vertices contained in
	bags of inner vertices of the arborescence into leaf bags, and then
	transform the result into a cubic tree.
	
	\begin{definition}[Leaf Directed Tree Decomposition]
		A directed tree decomposition $\br{T,\beta,\gamma}$ of a digraph~$D$ is called a \emph{leaf directed tree decomposition} if  $\fkt{\beta}{t} = \emptyset$ for all $t \in \V{T} \setminus \Leaves{T}$.
	\end{definition}
	
	So first, we show that a directed tree decomposition can be turned
	into a leaf decomposition without changing its width.
	
	% Let us call a directed tree decomposition $(T,\beta,\gamma)$ of a
	% digraph $D$ \emph{strong} if for all $t\in V(T)$, $\bigcup_{t'>t} \beta(t')$
	% induces a strongly connected subgraph of $D$. The following lemma is
	% folklore, we give a proof for completeness.
	
	% \begin{lemma}\label{lem:strong_dtDec}
	%   Every directed tree decomposition of a digraph $D$ can be converted
	%   in linear time into a strong directed tree decomposition.
	% \end{lemma}
	% \begin{proof}
	%   The construction is a bottom-up induction along $T$. Let $t$ be a
	%   vertex of $T$ and assume that for all $t'>t$ the condition is
	%   already satisfied. If $\beta(t)=\emptyset$, we are done. Otherwise
	%   let $t_1,\ldots,t_\ell$ be the children of $t$, let $C_1,\ldots,C_m$
	%   be all strongly connected components of
	%   $\beta(t)\cup \bigcup_{t'>t}\beta(t')$ and for $1\le i \le m$ let
	%   $B_i \coloneqq C_i \cap \beta(t)$. Note that by the induction
	%   hypothesis $\ell \le m$. We replace $t$ by a directed path
	%   $b_1,\ldots,b_m$ such that the parent of $t$ in $T$ has an edge to
	%   $t_1$. Let $\beta(b_i) \coloneqq B_i$ and let
	%   $\gamma(b_i) = \gamma(t)$. We make an edge from $b_i$ to all $t_j$
	%   such that $\beta(t_j) \subseteq B_i$ where $1\le i\le m$ and
	%   $1\le j\le \ell$. It is straightforward to see that we constructed a
	%   strong directed tree decomposition of the same width as
	%   $(T,\beta,\gamma)$ and that the construction can be computed in time
	%   linear in the size of $(T,\beta,\gamma)$.
	% \end{proof}
	
	\begin{lemma}\label{lem:leaf_decomp}
		Let $(T,\beta,\gamma)$ be a directed tree decomposition of a digraph
		$D$. There is a linear time algorithm that computes a leaf directed 
		tree decomposition of $D$ of the same width.
	\end{lemma}
	\begin{proof}
		For every inner vertex $t \in \V{T}$ such that
		$\fkt{\beta}{t}\neq \emptyset$ we add a new leave $t'$
		adjacent to $t$ (and no other vertices of $T$) and thus obtain a
		new tree $T'$. The new bags are defined by $\beta'\coloneqq
		V(T') \to 2^{V(D)}$ with $\beta'(t) \coloneqq \beta(t)$ for
		$t\in L(T)$, and $\beta'(t) \coloneqq \emptyset$ and $\beta'(t')
		\coloneqq \beta(t)$ for $t\in V(T)\setminus L(T)$. The new
		guards are defined by
		\begin{align*}
			\fkt{\gamma'}{e} \coloneqq
			\begin{cases}
				\fkt{\gamma}{e} & \text{ if } e\in \E{T}\\
				\beta(t') & \text{ if } e = \br{t,t'} \text{
					for some } t\in V(T)\setminus L(T).
			\end{cases}
		\end{align*}
		We prove that $\br{T',\beta',\gamma'}$ is still a directed \treeDecomp.
		The bags given by $\beta'$ still provide a near partition of $\V{D}$.
		For all edges $e \in \E{T}$ it is still the case that $\bigcup\condset{\fkt{\beta}{s}}{s\in\V{T'},s > e}$ is $\fkt{\gamma}{e}$-normal and for the new edges the normality is obvious.
		Finally, if $\Gamma'$ is defined for $(T',\beta',\gamma')$ as $\Gamma$ for $(T,\beta,\gamma)$, then $\Gamma'(t) = \Gamma(t)$ for all $t\in V(T)$ and for all $t'$ we have $\Gamma'(t') \subseteq \Gamma(t)$. Hence, the width of the decomposition did not change. Clearly, the new decomposition can be computed in linear time.
	\end{proof}
	
	So whenever we are given a directed tree decomposition of a digraph $D$, we can manipulate it such that exactly the leaf-bags are non-empty.
	This is still not enough since a cycle decomposition requires every leaf to be mapped to exactly one vertex -- so every bag has to be of size at most one -- and also the decomposition tree itself has to be cubic.
	The following lemma shows how a leaf directed tree
        decomposition can be further manipulated to meet the above
        requirements, again in polynomial time and without changing
        the width. We call a directed tree decomposition subcubic if
        its arborescence is subcubic.
	
\begin{lemma}
  \label{lem:cubic-leaf-decomp}
  Let $D$ be a digraph.  If there exists a directed \treeDecomp of
  width~$k$ for $D$, there also exists a subcubic leaf directed
  \treeDecomp of width $k$ for $D$ where every bag has size at most one.
\end{lemma}	
\begin{proof}
  Let $D$ be a digraph and $\br{T'',\beta,\gamma}$ a directed
  \treeDecomp of width $k$.  Then, due to \cref{lem:leaf_decomp}, there
  exists a leaf directed \treeDecomp~$\br{T,\beta,\gamma}$ of $D$ of
  the same width.  \Cref{alg:cubify} takes this as input and
  transforms it into a subcubic leaf directed
  \treeDecomp~$\br{T',\beta',\gamma'}$.
		
  \begin{algorithm}[h!]
    \caption{cubify a leaf decomposition}\label{alg:cubify}
    \begin{algorithmic}[1]
      \Procedure{cubify}{$\br{T,\beta,\gamma}$}
      \State $T' \gets T$, $\beta' \gets \beta$, $\gamma' \gets \gamma$
      \ForAll{$l \in \Leaves{T}$}
        \State $x \gets$ parent of $l$
        \ForAll{$v \in \fkt{\beta}{l}$}
          \State introduce new vertex $l_v$
          \State $T' \gets T' + \br{l,l_v}$
          \State $\fkt{\beta'}{l_v} \gets \Set{v}$
          \State $\fkt{\gamma'}{l,l_v} \gets \fkt{\beta}{l} \cup
          \fkt{\gamma}{x,l}$
        \EndFor
        \State $\fkt{\beta'}{l} \gets \emptyset$
      \EndFor
      \While{$T'$ not subcubic}
      \State let $t \in \V{T}$ with $\deg(t) = d+1 > 3$ ($\deg(t) =
      \Abs{\{e\in E(t) \mid e\sim t\}}$)
      \State $x \gets$ parent of $t$
      \State let $c_1,\dots,c_d$ be the
      children of $t$ in topological order of their bags\footnote{Note
      that all bags have size at most $1$.}
      \State
      introduce new vertices $t_1,\dots,t_{d-1}$ with empty bags
      \State
      $T' \gets T' - t + \Set{t_1,\dots,t_{d-1}} + \Set{\br{t_i,c_i}
        \mid 1 \leq i \leq d-1} + \Set{\br{t_i,t_{i+1}} \mid 1 \leq i
        \leq d-2} + \br{t_{d-1},c_d} + \br{x,t_1}$
      \State
      $\fkt{\gamma'}{t_i,c_i} \gets \fkt{\gamma}{t,c_i}$ for all
      $1 \leq i \leq d-1$
      \State $\fkt{\gamma'}{t_{d-1},c_d} \gets \fkt{\gamma}{t,c_d}$
      \State $\fkt{\gamma'}{t_{i},t_{i+1}}, \fkt{\gamma'}{x,t_{1}} \gets
      \fkt{\gamma}{x,t}$
      \EndWhile
      \State \textbf{return}
      $\br{T',\beta',\gamma'}$
      \EndProcedure
    \end{algorithmic}
  \end{algorithm}
		
  Clearly the resulting tree $T'$ is subcubic and only the leave bags
  contain vertices.  Now we want to check whether the output of the
  algorithm again yields a proper directed \treeDecomp of desired width.
		
  In the first part we split the bag of each leaf up into bags of
  single vertices which are added as new children.  We show that such
  a split of a leaf $l$ does not destroy the properties of the
  directed \treeDecomp.  The new vertices $l_v$ obtain bags of size
  $1$.  The new edge $\br{l,l_v}$ obtains the guard
  $\fkt{\beta}{l} \cup \fkt{\gamma}{x,l_v}$.
		
  Due to
  $\Abs{\fkt{\beta}{l} \cup \fkt{\gamma}{x,l_v} \cup \Set{v}} =
  \Abs{\fkt{\beta}{l} \cup \fkt{\gamma}{x,l}} \leq k+1$, the width of
  the new decomposition is still at most $k$.
		
  The guard of the edge going to $l_v$ contains $v$, therefore every
  walk in $D$ starting and ending at $v$ intersects the guard.  So
  after the first part of the algorithm $\br{T',\beta',\gamma'}$ is
  still a proper directed \treeDecomp.
		
  In the second part we split high degree vertices into paths.  For
  every vertex~$t$ of (total) degree $d+1>3$ we introduce $d-1$ new
  vertices $t_1,\dots,t_{d-1}$.  Let $c_1,\dots,c_d$ be the children
  of $t$.  We can assume without loss of generality that the children
  are ordered by the topological order of their bags.  That is if
  $i < j$, then every path from $\fkt{\beta}{T_{c_j}}$ to $\fkt{\beta}{T_{c_i}}$
  intersects $\fkt{\Gamma}{v}$.  The subtrees rooted at the children
  stay intact and are attached differently to the subtree above
  $T-T_t$.  To do this we first remove $t$ from $T$ obtaining subtrees
  $T_r$ containing the root and the parent~$x$ of $t$ as a leaf, and
  $T_{c_i}$ for every child of $t$.  We now add the new vertices as
  follows.  The former parent $x$ builds a path with the new vertices
  $t_i$ in increasing order.  Then every $t_i$ is mapped to the
  corresponding $c_i$, leaving $c_d$ which is also mapped to $t_{d-1}$
  which only has two neighbours so far, since its the last on the
  path.
		
  For all the subtrees that stay the same during the construction it
  is clear that no walk can leave them and come back without
  intersecting a guard.  But we introduce new subtrees that contain
  several child-subtrees of $t$.  Let $T'_{t_i}$ be such a subtree.
  Assume there is a walk $W$ in $D$ starting and ending in
  $\fkt{\beta}{T_{t_i}}$ containing a vertex from
  $\fkt{\beta}{T-T_{t_i}}$ but no vertex of $\fkt{\gamma}{t}$, which
  is the guard for the edge towards $t_i$.  There are two
  possibilities.  Either $W$ contains a vertex of $T'-T_{t_1}=T-T_t$,
  which directly yields a contradiction to $\br{T,\beta,\gamma}$ being
  a proper directed \treeDecomp.  Or $W$ contains a vertex from
  $\fkt{\beta}{T_{t_j}}$ for some $j < i$.  But this would imply that
  there is a path from $\fkt{\beta}{T_{c_i}}$ to
  $\fkt{\beta}{T_{c_j}}$, which contradicts the topological ordering.
		
  Thus, the output of the algorithm is again a proper directed
  \treeDecomp.
\end{proof}
	
	So given a directed \treeDecomp of a digraph $D$, we can transform it into a subcubic leaf directed tree decomposition $\br{T,\beta,\gamma}$ of $D$ in linear time without changing its width.
	It remains to show that if we forget about the orientation of the edges of the arborescence $T$, $\br{T,\beta}$ defines a cycle decomposition of bounded width.
	
\begin{proposition}
\label{lem:cyw-leq-dtw}
For every \digraph $D$ we have $\cycWidth{D} \leq 2\dtw{D}$.
\end{proposition}
\begin{proof}
  Let $k \coloneqq \dtw{D}$.  Due to \cref{lem:cubic-leaf-decomp}
  there exists a subcubic leaf directed \treeDecomp
  $\br{T,\beta,\gamma}$ of $D$ of width $k$ such that every
  leaf bag contains at most one vertex.  We want to show that
  $\br{T,\beta}$ yields a cycle decomposition of width at most $2k$.
  The function $\beta$ already provides a bijection between
  $\Leaves{T}$ and $\V{D}$.  So next we show that every edge
  $e \in \E{T}$ satisfies $\cycPor{e} \leq 2\Abs{\fkt{\gamma}{e}}$.
  Afterwards we make the subcubic decomposition cubic.
		
  Let $\mathcal{C}$ be a minimal family of pairwise disjoint directed
  cycles in $D$ and let $e\in E(T)$.  We show that
  $\Abs{\cut{e}\cap E(\mathcal{C})} \le 2\Abs{\gamma(e)}$.  Let
  $X_1,X_2\subseteq\V{D}$ be the two shores of the cut $\cut{e}$ such
  that $X_1=\fkt{\beta}{T'}$ where $T'\subseteq T$ is the subtree of
  $T$ not containing the root.  Furthermore, let
  $Y_1\subseteq\V{\mathcal{C}}\cap X_1$ be the vertices of the cycles
  in $\mathcal{C}$ incident with an edge of
  $\cut{e}\cap\E{\mathcal{C}}$.

  Let $\mathcal{W}$ be the collection of directed walks $W_{v,w}$ from
  a vertex $v\in Y_1$ to some $w\in Y_1$ such that
\begin{enumerate}
\item $W_{v,w}$ is a subwalk of some cycle in $\mathcal{C}$, and  
\item $\emptyset\neq\V{W_{v_1}}\setminus\Set{v_1,v_2}\subseteq X_2$.
\end{enumerate}
In other words, $\mathcal{W}$ is the set of walks (paths or cycles)
starting in some $v\in Y_1$ and going along the cycle in $\mathcal{C}$
that contains $v$ and ending in the first vertex in $X_1$ after
leaving it from $v_1$.  Clearly, the walks in $\mathcal{W}$ are not
necessarily vertex disjoint as the paths may share common endpoints in
$Y_1$.
		
Let $W_1,W_2\in\mathcal{W}$ be two walks with
$\V{W_1}\cap\V{W_2}\neq\emptyset$, then there is a cycle
$C\in\mathcal{C}$ such that both $W_1$ and $W_2$ are subwalks of $C$.
Hence $\Abs{\V{W_1}\cap\V{W_2}}\leq 2$.
As $\fkt{\gamma}{e}$ is a guard in the
directed \treeDecomp, it must contain a vertex of every walk in
$\mathcal{W}$. Every vertex can guard at most two paths, hence
$\cycPor{\cut{e}} = \Abs{\mathcal{W}} \leq 2\Abs{\fkt{\gamma}{e}}$.

% Moreover,
% $\Abs{\V{W_1}\cap\V{W_2}}= 2$ if and only if $W_1+W_2=C$.
		
% We can partition $\mathcal{W}$ into three classes.  Let
% $\mathcal{W}_0\subseteq\mathcal{W}$ be the set of walks
% $W\in\mathcal{W}$ that are disjoint from all other walks in
% $\mathcal{W}$.  Note that this includes all cycles in $\mathcal{W}$.  By
% $\mathcal{W}_2$ we denote the collection of walks $W_1\in\mathcal{W}$
% for which a walk $W_2\in\mathcal{W}\setminus\Set{W_1}$ exists, such
% that $\Abs{\V{W_1}\cap\V{W_2}}= 2$.  At last, let
% $\mathcal{W}_1\coloneqq
% \mathcal{W}\setminus\br{\mathcal{W}_0\cup\mathcal{W}_2}$ be the set of
% walks that intersect at least one other walk in $\mathcal{W}$ in
% exactly one vertex.
% We denote the sizes of the sets $\mathcal{W}_i$ by $w_i$ for
% $i\in\Set{0,1,2}$. 
		
% As the walks in $\mathcal{W}_0$ do not intersect, we need at least one
% vertex of $\gamma(e)$ to destroy one path.  Furthermore, one vertex
% can be contained in at most two walks of $\mathcal{W}_1$ and the
% same holds for $\mathcal{W}_2$.  Thus
% $\Abs{\fkt{\gamma}{e}} \ge w_0 + \frac{w_1+w_2}{2}$, so
% $\cycPor{\cut{e}} = w_0+w_1+w_2\leq 2\Abs{\fkt{\gamma}{e}}$.

Now note that there still might be vertices of degree two in $T$.
Since they are not leaves, $\beta$ does not map them to any vertex of
$\V{G}$.  Therefore, the two edges incident to a vertex of degree two
induce the same cut and we can contract one of them to reduce the
number of vertices of degree two.  Let $\br{T',\beta}$ be the
decomposition obtained in this way, then $\br{T',\beta}$ is cubic and
all cuts induced by edges still have porosity at most $2k$.  Thus,
$\br{T',\beta}$ is a cycle decomposition of $D$ of width at most $2k$.
\end{proof}
	
Hence, given a directed \treeDecomp of width $k$, we can compute a
cycle decomposition of width at most $2k$ in polynomial time using the
algorithms from the proofs of
\cref{lem:leaf_decomp,lem:cubic-leaf-decomp}.  The proofs of other
known obstructions to directed \treewidth such as \emph{well-linked
  sets} or \emph{havens} (see
\cite{reed1999directed,johnson2001directed}) yield constant factor
approximation fixed-parameter tractable algorithms for directed
\treeDecomps as described in \cite{dehmer2014quantitative}. (Given a
computational problem and a parameter (in our case, directed
treewidth), an algorithm solving it is fixed-parameter tractable if
its running time is bounded by a function of the form
$f(k)\cdot n^{O(1)}$ where $f$ is a computable function.) If we combine
this with our results from above we obtain the following concluding
corollary.
	
\begin{corollary}\label{cor:constantFactorCW}
  There is a fixed-parameter tractable approximation algorithm for \cyclewidth.
\end{corollary}
	
Please note that at this point we cannot make any statement on the
quality of the approximation since we do not know lower bounds for
\cyclewidth.
	
\subsection{A Grid Theorem for \Cyclewidth}
	
As we have seen in the previous subsection, \cyclewidth is bounded
from above by directed \treewidth.  The goal of this subsection is to
establish a lower bound.  Here we face a special challenge.  While
most width parameters, including directed \treewidth, imply
separations of bounded size, namely in the width of the decomposition,
\cyclewidth does not immediately imply the existence of a bounded size
separation.  Moreover, it is not immediately clear whether there
exists a function $f\colon\N\rightarrow\N$ such that, given a digraph
$D$ and a cut $\cut{X}$ in $D$, there is always a set
$S\subseteq\fkt{V}{D}$ that hits all directed cycles crossing
$\cut{X}$ and satisfying $\Abs{S}\leq\fkt{f}{\cycPor{\cut{X}}}$.
	
We will take a different approach.  To show that directed \treewidth
poses, qualitatively, as a lower bound for \cyclewidth, it suffices to
show that an obstruction for directed \treewidth also gives a lower
bound on \cyclewidth.  This will imply that any graph of large
directed \treewidth must also have high \cyclewidth.
	
In order to establish such a result we will show that the \cyclewidth
of a \digraph $D$ is an upper bound on the \cyclewidth of any
butterfly minor $H$ of $D$.  Then, in a second step we will show that
the \cyclewidth of the cylindrical grid depends on its order.  By
using the Directed Grid Theorem and \cref{lem:cyw-leq-dtw} we will
then obtain the equivalence of \cyclewidth and directed \treewidth as
desired.
	
\begin{theorem}
  \label{thm:cw_minor_closed}
  If $D$ is a \digraph and $D'$ is a butterfly minor of $D$, then
  $\cycWidth{D'} \leq \cycWidth{D}$.
\end{theorem}
\begin{proof}
  We first note that the \cyclewidth is closed under taking subgraphs.
  To see this let $D' \subseteq D$ be a subgraph of $D$ and
  $(T,\varphi)$ a cycle decomposition of $D$.  We delete every leaf
  that corresponds to a vertex in $\V{D} \setminus \V{D'}$ and
  eliminate vertices of degree two if needed as in the proof of
  \cref{lem:cyw-leq-dtw} to obtain a new cycle decomposition
  $(T',\varphi')$ of $D'$ whose maximal cycle porosity is at most the
  maximal cycle porosity of $(T,\varphi)$.  So
  $\cycWidth{D'} \leq \cycWidth{D}$.
		
  Next, we want to show that butterfly contracting an edge in $D$ does
  not increase the \cyclewidth.  Let $D' \coloneqq D/e$ for some edge
  $e = \br{u,v} \in \E{D}$.  Since $e$ is butterfly contractible it is
  the only outgoing edge from $u$ or the only ingoing edge of $v$. We
  assume the former case first. Note that every cycle containing $u$ also
  contains $v$.
		
  We obtain $(T',\varphi')$ from $(T,\varphi)$ by deleting the leaf
  $\ell$ of $T$ mapped to $u$ and contracting one of the two edges in
  $T$ incident with the unique neighbour of $\ell$ in $T$ in order to
  obtain the cubic tree $T'$.  Let $x_{u,v}$ be the contraction
  vertex, we set $\fkt{\phi'}{\fkt{\phi^{-1}}{v}}\coloneqq x_{u,v}$
  while leaving the mapping of the other leaves intact.
		
  All cuts in the decomposition for which $u$ and $v$ lie on the same
  side do not change their porosity.  So consider a cut $\cut{X}$,
  $v\in X$, induced by $(T,\varphi)$ that separates $u$ and~$v$, then
  $\br{T',\varphi'}$ induces a cut $\cut{X'}$ where
  $X'=\br{X\setminus\Set{v}}\cup\Set{x_{u,v}}$ and
  $\V{D/e}\setminus X'=\V{D}\setminus \br{X\cup\Set{u}}$.
		
  Suppose there is a family of pairwise disjoint directed cycles
  $\mathcal{C}$ in $D$ that contains a cycle~$C$ with $u\in\fkt{V}{C}$
  and satisfies $\Abs{\cut{X}\cap\E{\mathcal{C}}}=\cycPor{\cut{X}}$ as
  well as $\cut{X}\cap\E{C}\neq\emptyset$.  Let $C'$ be the cycle in
  $D/e$ obtained from $C$ after the contraction of $e$.
		
  Let $\mathcal{C'}$ be a family of directed cycles in $D/e$. At most
  one cycle $C$ contains $x_{u,v}$. If $C$ also exists in $D$, then
  $\mathcal{C'}$ is a family of cycles in $D$ as well. Otherwise, as
  $(u,v)$ is the only edge leaving $u$, the predecessor of $x_{u,c}$
  on $C$ has an edge to $u$ in $D$. Thus we can construct a cycle $C'$
  in $D$ from $C$ by replacing the edge $(y,x_{u,v})$ by a path
  $y,u,v$. Then $C'$ crosses $\cut{X}$ at least as often as $C$ and
  the cycle porosity of $\br{\mathcal{C}\setminus (C)}\cup \Set{C'}$
  is at least as high as the porosity of $\mathcal{C}$.
		
  For handling the case that the edge $e$ is the only ingoing edge of
  $v$, we show that the \cyclewidth does not change if we reverse all
  directions of the edges in the graph.  That is because we still get
  exactly the same cycles just with reversed direction and they still
  cross the same cuts.  Therefore the decomposition stays exactly the
  same with the same porosities for all cuts.
		
  By these arguments $\cycWidth{D'} \leq \cycWidth{D}$ holds for every
  butterfly minor $D'$ of $D$.
\end{proof}
	
The reverse direction follows from the Directed Grid Theorem.  It says
that, whenever the directed \treewidth of a digraph is large enough,
one can find a specific butterfly minor of large width.
	
In order to use the Directed Grid Theorem for our purposes we need to
show that the cylindrical grid has unbounded \cyclewidth.  It actually
suffices to prove a statement that gives a lower bound on the
\cyclewidth of a cylindrical grid of order $n$ in terms of its order.
To do this we have to analyse how the cycles in such a grid behave
when a reasonably large part of it is separated from the rest.
	
We call a cut $\cut{X}$ in a digraph $D$ \emph{balanced} if
$\Abs{X}\geq \frac{\Abs{\V{D}}}{3}$ and
$\Abs{\V{D}\setminus X}\geq\frac{\Abs{\V{D}}}{3}$.
	
\begin{lemma}
  \label{thm:cyl_grid_high_cw}
  The cylindrical grid of order $k$ has \cyclewidth at least
  $\frac{2}{3}k$.
\end{lemma}
\begin{proof}		
  Let $\CylGrid{k}$ be the cylindrical grid of order $k$ and let
  $(T,\varphi)$ be an optimal cycle decomposition of $\CylGrid{k}$.
		
  First we need to show that $T$ contains an edge such that the cut
  induced by this edge is balanced in $\CylGrid{k}$.  Such a cut can
  be found as follows.  We direct every edge of $T$ such that it
  points in direction of the subtree of $T$ containing more vertices
  of $\CylGrid{k}$.  If both sides contain exactly half the vertices
  we found the balanced cut.  Otherwise every edge can be directed and
  no two edges can point away from each other.  Also all leaf edges
  point away from the leaf.  So there has to be an inner vertex $v$
  with only ingoing edges, this vertex defines three subtrees of $T$,
  each separated from $v$ by one of its three incident edges.
  Each two subtrees contain together at least half of the
  vertices.  If there is a subtree with less than one third of the
  vertices, then the other two edges induce balanced cuts.  Otherwise
  all three subtrees contain exactly one third of the vertices and all
  three edges induce balanced cuts.
		
  Consider such a balanced cut $\Cut=\cut{e}$ in $\CylGrid{k}$ induced
  by an edge $e\in\E{T}$.  There are two cases: either each shore of
  $f$ contains one of the concentric cycles of $\CylGrid{k}$ or one of
  its shores does not contain any of the concentric cycles of the grid
  completely.
		
  In the case where each shore of $f$ contains one of the concentric
  cycles of $\CylGrid{k}$ we are able to construct a cycle $C$ that
  contains $2k$ edges of $f$.  Let $C_{\text{in}}$ be one of the
  concentric cycles of $\CylGrid{k}$ that is completely on one of the
  shores of $f$, which we will call the inner shore, and
  $C_{\text{out}}$ a concentric cycle of $\CylGrid{k}$ on the other
  shore.  We call the shortest paths from $C_{out}$ to $C_{in}$ ingoing
  and the shortest paths in the other direction outgoing.
		
  We start on a vertex of $C_{\text{out}}$ where it intersects an
  ingoing path of $\CylGrid{k}$.  Then we walk along the ingoing path
  until we meet $C_{\text{in}}$ and walk along it for an edge.  There
  we meet an outgoing path and walk along it until we intersect the
  outer cycle again.  This we repeat until we reach the starting
  vertex.  Since we used all in- and outgoing paths of $\CylGrid{k}$,
  there are $2k$ subpaths of $C$ crossing $\Cut$ at least once,
  therefore $C$ contains at least $2k$ edges of $\Cut$.  Thus $\Cut$
  has cycle porosity at least $2k > \frac{2}{3}k$.
		
  Next, consider the case that there is a shore of $f$ that does not
  contain any concentric cycle of $\CylGrid{k}$.  The other shore of
  $\Cut$ can contain at most two third of the concentric cycles of
  $\CylGrid{k}$ as $f$ is balanced.  Therefore, the remaining at least
  $\frac{k}{3}$ cycles of $\CylGrid{k}$ cross $\Cut$.  Since they are
  cycles, each of them meets $f$ in at least two edges and thus
  $\cycPor{\cut{e}}\geq\frac{2}{3}k$.  So, finally
  $\cycWidth{\CylGrid{k}} \geq \frac{2}{3}k$.
	\end{proof}
	
	Together with \cref{thm:cw_minor_closed} this implies the following corollary.
	
	\begin{corollary}\label{cor:gridminor_highCycW}
		If a digraph $D$ has the cylindrical grid of order $k$ as a butterfly minor, then its \cyclewidth is at least $\frac{2}{3}k$.
	\end{corollary}
	
	\begin{theorem}
		A class $\mathcal{D}$ of digraphs is a class of bounded directed \treewidth, if and only if it is a class of bounded \cyclewidth.
	\end{theorem}
	
	\begin{proof}
		Let $\mathcal{D}$ be a class of digraphs and let $f$
                be the function from \cref{thm:DGT}.
		Suppose $\mathcal{C}$ has unbounded directed \treewidth, then for each $n\in\N$ there is a digraph $D'_n\in\mathcal{D}$ such that $\dtw{D'}\geq f(n)$.
		By \cref{thm:DGT}, there is a digraph $D_n\in\mathcal{D}$ that contains the cylindrical grid of order $n$ as a butterfly minor.
		Therefore, $\cycWidth{D_n}\geq\frac{2}{3}n$ by
		\cref{cor:gridminor_highCycW} and thus $\mathcal{C}$ has also
		unbounded \cyclewidth. The other direction is \cref{lem:cyw-leq-dtw}.
	\end{proof}
	
	We conclude this section by reformulation of \cref{thm:DGT} into a grid theorem for \cyclewidth.
	This is a direct corollary of the main result of this section.
	
	\begin{theorem}\label{cor:cywgridthm}
		There is a function $f \colon\N \to \N$ such that every digraph $D$ either satisfies $\cycWidth{D}\leq\ f(k)$, or contains the cylindrical grid of order $k$ as a butterfly minor.
	\end{theorem}

	\section{Perfect Matching Width}
	\label{sec:pmw}
	
	We will now leave the world of directed graphs and consider undirected graphs with perfect matchings.
	As we have seen in \cref{obs:strongmdirections}, strongly connected
	directed graphs correspond to matching covered bipartite graphs with a
	fixed perfect matching. We discuss this correspondence in more detail
	in this section. 
	
	In this section we establish a connection between the perfect matching width of bipartite matching covered graphs and the directed \treewidth of their $M$-directions.
	
	This is done in two steps.
	First we introduce perfect matching width and relate it to the directed \treewidth of $M$-directions.
	Then, using the relation between matching minors and butterfly minors of $M$-directions of bipartite graphs, we deduce the Bipartite Matching Grid Theorem.
	
	\subsection{Perfect Matching Width and Directed Cycles}
	
	\begin{definition}[matching-porosity]
		Let $G$ be a matching covered graph and $X\subseteq\V{G}$.
		We define the \emph{\matPoros} of $\cut{X}$ as follows:
		\begin{align*}
			\matPor{\cut{X}} \coloneqq \max_{\Matching \in \perf{G}} \Abs{\Matching \cap \cut{X}}.
		\end{align*}
	\end{definition}
	A perfect matching $M\in\perf{G}$ is \emph{maximal with respect to a cut $\cut{X}$} if there is no perfect matching~$\Matching' \in \perf{G}$ such that $\cut{X} \cap \Matching \subsetneq \cut{X} \cap \Matching'$.
	
	A perfect matching~$\Matching \in \perf{G}$ \emph{maximises} a cut $\cut{X}$ if $\matPor{\cut{X}} = \Abs{\Matching \cap \cut{X}}$.
	
	\begin{definition}[Perfect Matching Width]
		\label{def:pmw}
		Let $G$ be a matching covered graph.
		A perfect matching decomposition of $G$ is a
                tuple~$\br{T,\delta}$, where $T$ is a cubic tree and
                $\delta\colon \Leaves{T} \to \V{G}$ is a bijection.
		The width of $\br{T,\delta}$ is given by $\max_{e \in \E{T}} \matPor{e}$ and the \emph{perfect matching width} of~$G$ is then defined as 
		\begin{align*}
			\pmw{G} \coloneqq \min_{\substack{\br{T,\delta} \text{ perfect matching}\\\text{decomposition of } G}} \quad \max_{e \in \E{T}} \matPor{\cut{e}}.
		\end{align*}
	\end{definition}
	
	% The goal of this subsection is to prove a relation between perfect matching width and directed \treewidth.
	% As a first step we relate perfect matching width to \cyclewidth as introduced earlier.
	
	% The problem with general perfect matching decompositions of matching covered graphs is, that, given a fixed perfect matching $M\in\perf{G}$ for some matching covered graph $G$, the endpoints of the edges of $M$ can be arbitrarily far apart from each other in the decomposition tree.
	% However, if $G$ is bipartite and we want to relate the width of a perfect matching decomposition of $G$ to the \cyclewidth of its $M$-direction, we need these vertices to be close.
	
	% So what we need to be able to do is manipulating a perfect matching decomposition with respect to a fixed perfect matching without significantly changing the width.
	
	% The following lemma shows that if we move a single vertex from one shore of a cut to the other, then the porosity of the cut increases by at most one.
	% This especially means that moving a vertex over a cut in a perfect matching decomposition increases the width by at most one.
	
	As every vertex is contained in exactly one perfect matching edge, we
	obtain the following observation. 
	\begin{observation}
		\label{cor:moveVertex_PMWchangeSmall}
		For all matching covered graphs $G$, for all $\Shore \subseteq \V
		{G}$ and all $x \in \V{G}\setminus\Shore$,
		$\matPor{\cut{\Shore}} - 1 \le \matPor{\cut{\Shore \cup \Set{x}}} \le \matPor{\cut{\Shore}} + 1$.
	\end{observation}

	If we are given a perfect matching $M\in\perf{G}$ of a matching covered graph $G$ and a cut $\cut{X}$ of matching porosity $k$, then there are at most $k$ vertices in $X$ that are incident with edges in $M\cap\cut{X}$.
	Hence we have to move at most $k$ vertices from one shore to the other in order to obtain a new cut where both shores are $M$-conformal.
	This leads to the following proposition.
	
	\begin{proposition}
		\label{prop:moveSetsOverCuts}
		Let $G$ be a matching covered graph, $\Shore \subseteq \V{G}$ and $\Matching \in \perf{G}$.
		Then there is an $\Matching$-\conformal set $\Shore' \subseteq \V{G}$ such that
		\begin{enumerate}
			\item $\Shore \subseteq \Shore'$,
			\item $\Abs{\Shore'} \leq \Abs{\Shore} + \matPor{\cut{\Shore}}$ and
			\item $\matPor{\cut{\Shore'}} \leq 2 \matPor{\cut{\Shore}}$.
		\end{enumerate}
	\end{proposition}
	
	Now, if we look at the $M$-direction of a matching covered bipartite graph $G$ with $M\in\perf{G}$, then any cycle decomposition $\br{T,\varphi}$ of $\dirm{G}{M}$ can be interpreted as a decomposition of $G$ where $\varphi$ is a bijection between $\Leaves{T}$ and $M$.
	Then every edge in $T$ induces a bipartition of $\V{G}$ into $M$-conformal sets.
	The next definition relates this observation to perfect matching decompositions.
	
	\begin{definition}[$M$-Perfect Matching Width]
		Let $G$ be a matching covered graph and $\Matching \in \perf{G}$.
		The \emph{$\Matching$-perfect matching width}, $\Mpmw{\Matching}{}$, is defined as the smallest width of a perfect matching decomposition of $G$ such that for every inner edge $e$ holds if $\br{T_1,T_2} = \splitTree{T}{e}$, then $\delta\br{\Leaves{T_1}}$ and $\delta\br{\Leaves{T_2}}$ are $\Matching$-\conformal.
	\end{definition}
	
	\begin{proposition}
		\label{prop:pmw_Mpmw}
		Let $G$ be a matching covered graph and $\Matching \in \perf{G}$.
		Then $\pmw{G} \leq \Mpmw{\Matching}{G} \leq 2\pmw{G}$.
	\end{proposition}
	
	\begin{proof}
		Clearly $\pmw{G} \leq \Mpmw{\Matching}{G}$, because the $\Mpmw{\Matching}{}$ is the width of a perfect matching decomposition.
		
		Next, we prove $\Mpmw{\Matching}{G} \leq 2\pmw{G}$.  Let
		$\br{T,\delta}$ be a perfect matching decomposition of $G$ of
		minimum width. Now let $X \subseteq \V{\Matching}$ such that
		for all $e \in \Matching$ holds $\Abs{e \cap X}=1$. Denote by
		$x\Matching$ the vertex $y$ with $xy\in \Matching$ and let
		$X'\subseteq X$ be the set of vertices $x\in X$ such that
		the path from $\delta^{-1}(x)$ to $\delta^{-1}(xM)$ in $T$ contains an inner
		edge (i.e.\@ an edge not incident with a leaf).
		
		Now we construct a new decomposition $\br{T',\delta'}$.
		We remove $\delta^{-1}(x\Matching)$ and add two new leaves to the vertex $\delta^{-1}(x)$ in $T'$.
		Moreover, the deletion of $\delta^{-1}(x\Matching)$ left a vertex of degree $2$, in order to maintain a cubic tree we contract one of the two edges incident with said degree $2$ vertex.
		Now $\delta^{-1}(x)$ has two new neighbours $a$ and $b$ which we map to the vertices $x\Matching$ and $x$ such that $\delta'(a) \coloneqq x\Matching$ and $\delta'(b) \coloneqq x$.
		The vertex~$\delta^{-1}(x)$ now is an inner vertex of $T'$
		therefore $\delta'$ is not defined on it. 
		
		The only additional inner edges in $T'$ are those where the corresponding cut separates a pair of leaves mapped to a matching edge of $\Matching$ containing a vertex in $X'$ from the rest of the graph.
		So these induce cuts of matching porosity at most 2 and $\Matching$-\conformal shores.
		
		Now consider an inner edge $e$ from $T$ and the two shores
		$\Shore$ and $\V{G}\setminus\Shore$ it induces.
		The edges of $\Matching' \subseteq \Matching$ that have vertices in both shores are at most $\pmw{G}$ many.
		Therefore by \cref{cor:moveVertex_PMWchangeSmall} the porosity of the induced cut is at most doubled.
	\end{proof}
	
	We now need the following observation.
	Let $G=\br{A\cup B,E}$ be a bipartite matching covered graph and $M,M'\in\perf{G}$ two distinct perfect matchings.
	Then the graph induced by $M\cup M'$ consists only of isolated edges and $M$-$M'$-conformal cycles.
	Moreover, the isolated edges are exactly the set $M\cap M'$.
	Let $C$ be such an $M$-$M'$-conformal cycle.
	Then in both, $\dirm{G}{M}$ and $\dirm{G}{M'}$, $C$ corresponds to a directed cycle.
	
	On the other hand let $N\in\perf{G}$ and let $C$ be a directed cycle in $\dirm{G}{N}$.
	Then $C$ corresponds to an $N$-conformal cycle $C_N$ in $G$ of exactly double the length, where $\E{C}$ and $\E{C_N}\setminus N$ coincide (up to the direction of the edges in $C$).
	Thus $\br{N\setminus\E{C_N}}\cup\br{\E{C_N}\setminus N}$ is a perfect matching of $G$.
	
	So there is a one-to-one correspondence between the directed cycles in $\dirm{G}{\Matching}$ and the $\Matching$-conformal cycles in $G$.
	Using this insight we can translate an $M$-perfect matching decomposition of $G$ to a cycle decomposition of $\dirm{G}{\Matching}$ and back.
	
	\begin{lemma}
		\label{lem:Mpmw_equ_cyw}
		Let $G$ be a bipartite and matching covered graph and $\Matching \in \perf{G}$.
		Then $\Mpmw{\Matching}{G} = \cycWidth{\dirm{G}{\Matching}}$.
	\end{lemma}
	\begin{proof}

		We first prove that $\Mpmw{\Matching}{G} \geq \cycWidth{\dirm{G}{\Matching}}$.
		Assume $\Mpmw{\Matching}{G} = k$ for some $k \in \N$.
		Then there is a perfect matching decomposition $\br{T,\delta}$ of width $k$ such that all shores of the cuts induced by inner edges are $\Matching$-conformal.
		We construct a cycle decomposition $\br{T',\varphi}$ of $D \coloneqq \dirm{G}{\Matching}$.
		In $\br{T,\delta}$ the leaves containing two vertices matched by $\Matching$ share a neighbour.
		We define $T' \coloneqq T-\Leaves{T}$. Recall that matching
		edges become vertices in $\dirm{G}{\Matching}$. For
		$xy\in\Matching$ let $t_{xy}$ be the common neighbour of
		$\phi^{-1}(x)$ and $\phi^{-1}(y)$. We define $\varphi(t_{xy})
		\coloneqq xy$.
		
		Now assume this decomposition has an edge $e \in T'$ that induces a cut $\cut{\Shore}$ of cycle porosity at least $k+1$.
		Then there is a family of directed cycles $\mathcal{C}$ in $D$ witnessing this.
		This corresponds to a family of $\Matching$-alternating cycles $\mathcal{C}'$ in $G$ that also has at least $k+1$ edges in the cut $\cut{\Shore'}$ induced by $e \in T$, note that $\V{T'}\subseteq \V{T}$.
		Since $\Shore'$ is $M$-\conformal, $\Matching \cap \br{E[\mathcal{C}'] \cap \cut{\Shore'}} = \emptyset$, that is none of the edges of $\mathcal{C}'$ that lie in the cut are from $\Matching$.
		Let $\Matching'$ be the matching we obtain by switching $\Matching$ along all the cycles in $\mathcal{C}'$, that is $\Matching' \coloneqq \br{\Matching \setminus E[\mathcal{C}']} \cup \br{E[\mathcal{C}'] \setminus \Matching}$.
		Now $\Matching'$ has at least $k+1$ edges in $\cut{e}$
		contradicting that $\br{T,\delta}$ has width $k$. Therefore $\br{T',\varphi}$ is a cycle decomposition of $D$ of width at most $k$.
		
		Now we prove that $\Mpmw{\Matching}{G} \leq \cycWidth{D}$.
		Let $\cycWidth{D} = k$ for some $k \in \N$.
		Then there is a cycle decomposition $\br{T,\varphi}$ of $D$ with width $k$.
		We construct a perfect matching decomposition $\br{T',\delta}$ of $G$.
		The construction basically works the other way around as in the first part of the proof.
		For every leaf in $T$ we introduce two new child vertices that are mapped to the two endpoints of the matching edge which is contracted into a vertex of contracted in $D$.
		Formally, $\V{T'} \coloneqq \V{T}\cup\bigl\{t_i \mid
		t \in \Leaves{T}, i\in \Set{\ell,r}\bigr\}$ and $\E{T'} \coloneqq \E{T}\cup\Set{tt_i
			\mid t \in \Leaves{T}, i\in\Set{\ell,r}}$ where all $t_r$ and
		$t_\ell$ are new vertices.
		Now for all $t\in \Leaves{T}$ if $\phi(t)$ is the vertex $xy
		\in\Matching$ of $D$,
		then $\delta(t_\ell) \coloneqq x$ and $\delta(t_r) \coloneqq y$.
		Since now all pairs of vertices that are matched by $\Matching$ have a common parent vertex in $T'$ the shores of the cuts induced by inner edges are $\Matching$-\conformal.
		Therefore the width of $\br{T',\delta}$ yields an upper bound on $\Mpmw{\Matching}{G}$.
		Assume there is an edge $e \in \E{T'}$ and a matching~$\Matching'$ such that $\Abs{\Matching' \cap \cut{e}} \geq k+1$.
		We consider the subgraph of $G$ only containing edges from~$\Matching$ and $\Matching'$.
		It consists only of disjoint cycles and independent edges.
		Because none of the edges in $\Matching$ lie in $\cut{e}$, all edges of $\Matching' \cap \cut{e}$ lie on $\Matching$-$\Matching'$-conformal cycles.
		Therefore there is a family of $\Matching$-conformal cycles with more than $k$ edges in $\cut{e}$.
		This corresponds to a family of directed cycles in~$D$ having more than $k$ edges in the cut induced by $e$ in~$D$.
		This yields a contradiction to $\br{T,\varphi}$ having width $k$.
		Therefore $\br{T',\delta}$ is a perfect matching decomposition of $G$ of width $k$.
	\end{proof}
	
	By combining \cref{prop:pmw_Mpmw} and \cref{lem:Mpmw_equ_cyw} we obtain the following result as an immediate corollary.
	
	\begin{theorem}
		\label{thm:pmw_cyw}
		Let $G$ be a bipartite and matching covered graph and $\Matching \in \perf{G}$.
		Then $\pmw{G} \leq \cycWidth{\dirm{G}{\Matching}} \leq 2\pmw{G}$.
	\end{theorem}
	
	The proof of \cref{lem:Mpmw_equ_cyw} provides an algorithm to translate a cycle decomposition of $\dirm{G}{M}$ into a perfect matching decomposition of $G$.
	The cycle decomposition itself can be computed in \FPT-time from a directed \treeDecomp by \cref{cor:constantFactorCW} and thus we obtain the following corollary.
	
	\begin{corollary}\label{cor:approx_pmw}
		There is an approximation algorithm for perfect matching width on bipartite matching covered graphs running in \FPT-time.
	\end{corollary}
	
	Please note here that the quality of the approximation provided by \cref{cor:approx_pmw} depends on the function of the lower bound for \cyclewidth in terms of directed \treewidth.
	Since the lower bound presented in this paper is essentially the function from the Directed Grid Theorem, we can only state that there is some function bounding the quality of this approximation.
	It would be interesting to see whether this can be turned into a constant factor.
	
	\subsection{The Bipartite Matching Grid}
	
	The standard concept of contractions in graphs reduces the number of vertices by exactly one.
	Thus it does not preserve the property of a graph to contain a perfect matching.
	However, if we always consider conformal subgraphs and contract two
	edges at a time,
	we can find a specialised version of minors that preserve the property
	of being matching covered.
	
	The idea of matching minors appears in the work of McGuaig~\cite{mccuaig2001brace}, but the formal framework and the actual name were introduced by Norine and Thomas in~\cite{norine2007generating}.
	
	\begin{definition}[Bicontraction]
		Let $G$ be a graph and let $v_0$ be a vertex of $G$ of degree two incident to the edges $e_1=v_0v_1$ and $e_2=v_0v_2$.
		Let $H$ be obtained from $G$ by contracting both $e_1$ and $e_2$ and deleting all resulting parallel edges.
		We say that $H$ is obtained from $G$ by \emph{bicontraction} or \emph{bicontracting the vertex $v_0$}.
	\end{definition}
	
	\begin{definition}[Matching Minor]
		\label{def:matching-minor}
		Let $G$ and $H$ be graphs.
		We say that $H$ is a \emph{matching minor} of $G$ if $H$ can be obtained from a \conformal subgraph of $G$ by repeatedly bicontracting vertices of degree two.	
	\end{definition}
	
	There is a strong relation between matching minors of bipartite
	matching covered graphs and butterfly minors of strongly connected digraphs.
	
	\begin{lemma}[McGuaig, 2000 \cite{mccuaig2000even}]\label{lemma:mcguigmatminors}
		Let $G$ and $H$ be bipartite matching covered graphs.
		Then $H$ is a matching minor of $G$ if and only if there exist  perfect matchings $M\in\perf{G}$ and $M'\in\perf{H}$ such that $\dirm{H}{M'}$ is a butterfly minor of $\dirm{G}{M}$.
	\end{lemma}
	
	We want to establish a relation between the perfect matching
        widths of
	matching covered graphs and their matching minors.
	At this point we do not know whether the perfect matching width in general is closed under matching minors.
	We are, however, able to at least show a qualitative closure by using our result on the relation of the \cyclewidth of digraphs and their butterfly minors.
	
	\begin{proposition}
		\label{thm:pmw_minor_closed}
		Let $G$ and $H$ be matching covered bipartite graphs.
		If $H$ is a matching minor of $G$, then $\pmw{H} \leq 2\pmw{G}$.
	\end{proposition}	
	\begin{proof}
		Let $H$ be a matching minor of $G$.
		Then \cref{lemma:mcguigmatminors} provides the existence of perfect matchings $M\in\perf{G}$ and $M'\in\perf{H}$ such that $\dirm{H}{M'}$ is a butterfly minor of $\dirm{G}{M}$.
		The $M$-perfect matching width of $G$ is at most $2\pmw{G}$ by \cref{prop:pmw_Mpmw} and, by \cref{lem:Mpmw_equ_cyw}, $\Mpmw{M}{G}=\cycWidth{\dirm{G}{M}}$.
		Since $\dirm{H}{M'}$ is a butterfly minor of $\dirm{G}{M}$, \cref{thm:cw_minor_closed} gives us $\cycWidth{\dirm{H}{M'}}\leq\cycWidth{\dirm{G}{M}}$.
		At last, using \cref{lem:Mpmw_equ_cyw} again and combining the
		above inequalities we obtain $\pmw{H}\leq2\pmw{G}$.
	\end{proof}
	
	As we are going for a cylindrical grid and derive the grid theorem for bipartite matching covered graphs from the Directed Grid Theorem anyway, it makes sense to derive our grid from the directed case as well.
	The following definition defines the bipartite matching grid
        by providing a procedure that allows us to obtain it from the
        directed cylindrical grid. Let $E_o$ be a set of edges of the
        outermost cycle containing every second edge. Let $E_i$ be the
        set of edges of the innermost cycle containing every second
        edge such that for every $e_i\in E_i$ there is an $e_o\in E_o$
        such that the shortest path from the tail of $e_i$ to the tail
        of $e_o$ has length $k-1$. The last condition assures that the
        chosen edges in both cycles are not ``shifted''.   
	
	\begin{definition}[Bipartite Matching Grid]
		\label{def:bipMatGrid}
		Let $k\in\mathds{N}$ be a positive integer.
		Let $\CylGridRetract{k}$ be the digraph obtained from
                $\CylGrid{k}$ by butterfly contracting every edge from
                $E_o$ and every edge from $E_i$.
		The \emph{bipartite matching grid} of order $k$ is the unique bipartite matching covered graph $\MatGrid{k}$ that has a perfect matching $M\in\perf{\MatGrid{k}}$ such that $\dirm{\MatGrid{k}}{M}=\CylGridRetract{k}$.
	\end{definition}
	
	The uniqueness of $\MatGrid{k}$ and $M$ follows from \cref{obs:strongmdirections}.
	As an example see \cref{fig:matgrid}.
	Here we construct $\MatGrid{3}$ from $\CylGrid{3}$.
	The edges from $E_o\cup E_i$ are marked.
	
	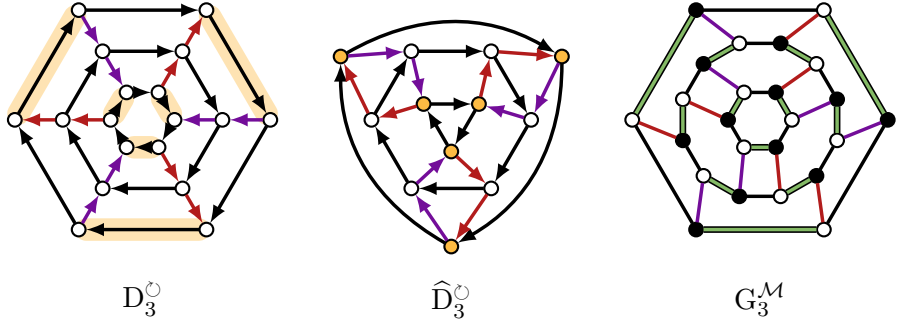
\begin{figure}[t]
		\begin{center}
			\begin{tikzpicture}[scale=0.7]
			
			\pgfdeclarelayer{background}
			\pgfdeclarelayer{foreground}
			
			\pgfsetlayers{background,main,foreground}
			
			%%%%% Vertex Styles %%%%%
			\tikzstyle{v:main} = [draw, circle, scale=0.5, thick,fill=white]
			\tikzstyle{v:border} = [draw, circle, scale=0.75, thick,minimum size=10.5mm]
			\tikzstyle{v:mainfull} = [draw, circle, scale=1, thick,fill]
			\tikzstyle{v:ghost} = [inner sep=0pt,scale=1]
			\tikzstyle{v:marked} = [circle, scale=1.5, fill=AO,opacity=0.3]
			%%%%% %%%%% %%%%%
			
			%%%%% Edge Styles %%%%%
			\tikzset{>=latex} 
			\tikzstyle{e:marker} = [line width=8.5pt,line cap=round,opacity=0.4,color=Dandelion]
			\tikzstyle{e:colored} = [line width=2.8pt,color=matEdgeColour]
			\tikzstyle{e:coloredthin} = [line width=1.1pt,opacity=0.8]
			\tikzstyle{e:coloredborder} = [line width=2pt]
			\tikzstyle{e:main} = [line width=1.3pt]
			\tikzstyle{e:extra} = [line width=1.3pt,color=LavenderGray]
			%%%%% %%%%% %%%%%
			
			\begin{pgfonlayer}{main}
			
			%%%%% Centered Ghost Vertices %%%%%
			\node (C) [] {};
			
			%%%%% Left Center %%%%%
			\node (C1) [v:ghost, position=180:58mm from C] {};
			%\node (U1) [v:ghost, position=90:50mm from C1] {};
			\node (L1) [v:ghost, position=270:34mm from C1,align=center] {$\CylGrid{3}$};
			
			%%%%% Center %%%%%
			\node (C2) [v:ghost, position=0:0mm from C] {};
			%\node (U2) [v:ghost, position=90:50mm from C2] {};
			\node (L2) [v:ghost, position=270:34mm from C2,align=center] {$\CylGridRetract{3}$};
			
			%%%%% Right Center %%%%%
			\node (C3) [v:ghost, position=0:58mm from C] {};
			%\node (U3) [v:ghost, position=90:50mm from C3] {};
			\node (L3) [v:ghost, position=270:34mm from C3,align=center] {$\MatGrid{3}$};
			%%%%% %%%%% %%%%%

			%%%%% Vertices %%%%%
			
			%%%%% Left Center %%%%%
			
			\node (C1v1) [v:main,position=180:6mm from C1] {};
			\node (C1v2) [v:main,position=120:6mm from C1] {};
			\node (C1v3) [v:main,position=60:6mm from C1] {};
			\node (C1v4) [v:main,position=0:6mm from C1] {};
			\node (C1v5) [v:main,position=300:6mm from C1] {};
			\node (C1v6) [v:main,position=240:6mm from C1] {};
			
			\node (C1v7) [v:main,position=180:15mm from C1] {};
			\node (C1v8) [v:main,position=120:15mm from C1] {};
			\node (C1v9) [v:main,position=60:15mm from C1] {};
			\node (C1v10) [v:main,position=0:15mm from C1] {};
			\node (C1v11) [v:main,position=300:15mm from C1] {};
			\node (C1v12) [v:main,position=240:15mm from C1] {};
			
			\node (C1v13) [v:main,position=180:24mm from C1] {};
			\node (C1v14) [v:main,position=120:24mm from C1] {};
			\node (C1v15) [v:main,position=60:24mm from C1] {};
			\node (C1v16) [v:main,position=0:24mm from C1] {};
			\node (C1v17) [v:main,position=300:24mm from C1] {};
			\node (C1v18) [v:main,position=240:24mm from C1] {};
			
			%%%%% %%%%% %%%%%
			
			%%%%% Center %%%%%
			
			\node(C2v1und6) [v:main,position=150:6mm from C2,fill=Dandelion] {};
			\node(C2v5und4) [v:main,position=270:6mm from C2,fill=Dandelion] {};
			\node(C2v3und2) [v:main,position=30:6mm from C2,fill=Dandelion] {};
			
			\node (C2v7) [v:main,position=120:15mm from C2] {};
			\node (C2v8) [v:main,position=60:15mm from C2] {};
			\node (C2v9) [v:main,position=0:15mm from C2] {};
			\node (C2v10) [v:main,position=300:15mm from C2] {};
			\node (C2v11) [v:main,position=240:15mm from C2] {};
			\node (C2v12) [v:main,position=180:15mm from C2] {};
			
			\node(C2v13und18) [v:main,position=150:24mm from C2,fill=Dandelion] {};
			\node(C2v17und16) [v:main,position=270:24mm from C2,fill=Dandelion] {};
			\node(C2v15und14) [v:main,position=30:24mm from C2,fill=Dandelion] {};
			
			%%%%% %%%%% %%%%%
			
			%%%%% Right Center %%%%%
			
			\node (C3v1) [v:main,position=120:6mm from C3] {};
			\node (C3v2) [v:main,position=60:6mm from C3,fill=black] {};
			\node (C3v3) [v:main,position=0:6mm from C3] {};
			\node (C3v4) [v:main,position=300:6mm from C3,fill=black] {};
			\node (C3v5) [v:main,position=240:6mm from C3] {};
			\node (C3v6) [v:main,position=180:6mm from C3,fill=black] {};
			
			\node (C3v7) [v:main,position=135:15mm from C3,fill=black] {};
			\node (C3v8) [v:main,position=105:15mm from C3] {};
			\node (C3v9) [v:main,position=75:15mm from C3,fill=black] {};
			\node (C3v10) [v:main,position=45:15mm from C3] {};
			\node (C3v11) [v:main,position=15:15mm from C3,fill=black] {};
			\node (C3v12) [v:main,position=345:15mm from C3] {};
			\node (C3v13) [v:main,position=315:15mm from C3,fill=black] {};
			\node (C3v14) [v:main,position=285:15mm from C3] {};
			\node (C3v15) [v:main,position=255:15mm from C3,fill=black] {};
			\node (C3v16) [v:main,position=225:15mm from C3] {};
			\node (C3v17) [v:main,position=195:15mm from C3,fill=black] {};
			\node (C3v18) [v:main,position=165:15mm from C3] {};
			
			\node (C3v19) [v:main,position=120:24mm from C3,fill=black] {};
			\node (C3v20) [v:main,position=60:24mm from C3] {};
			\node (C3v21) [v:main,position=0:24mm from C3,fill=black] {};
			\node (C3v22) [v:main,position=300:24mm from C3] {};
			\node (C3v23) [v:main,position=240:24mm from C3,fill=black] {};
			\node (C3v24) [v:main,position=180:24mm from C3] {};
			
			%%%%% %%%%% %%%%%
			
			%%%%% %%%%% %%%%%

			%%%%% Edges %%%%%
			
			%%%%% Left Center %%%%%
			
			\draw (C1v1) [e:main,->] to (C1v2);
			\draw (C1v2) [e:main,->] to (C1v3);
			\draw (C1v3) [e:main,->] to (C1v4);
			\draw (C1v4) [e:main,->] to (C1v5);
			\draw (C1v5) [e:main,->] to (C1v6);
			\draw (C1v6) [e:main,->] to (C1v1);
			
			\draw (C1v7) [e:main,->] to (C1v8);
			\draw (C1v8) [e:main,->] to (C1v9);
			\draw (C1v9) [e:main,->] to (C1v10);
			\draw (C1v10) [e:main,->] to (C1v11);
			\draw (C1v11) [e:main,->] to (C1v12);
			\draw (C1v12) [e:main,->] to (C1v7);
			
			\draw (C1v13) [e:main,->] to (C1v14);
			\draw (C1v14) [e:main,->] to (C1v15);
			\draw (C1v15) [e:main,->] to (C1v16);
			\draw (C1v16) [e:main,->] to (C1v17);
			\draw (C1v17) [e:main,->] to (C1v18);
			\draw (C1v18) [e:main,->] to (C1v13);
			
			\draw (C1v14) [e:main,color=myViolet,->] to (C1v8);
			\draw (C1v8) [e:main,color=myViolet,->] to (C1v2);
			\draw (C1v16) [e:main,color=myViolet,->] to (C1v10);
			\draw (C1v10) [e:main,color=myViolet,->] to (C1v4);
			\draw (C1v18) [e:main,color=myViolet,->] to (C1v12);
			\draw (C1v12) [e:main,color=myViolet,->] to (C1v6);
			
			\draw (C1v1) [e:main,color=myRed,->] to (C1v7);
			\draw (C1v7) [e:main,color=myRed,->] to (C1v13);
			\draw (C1v3) [e:main,color=myRed,->] to (C1v9);
			\draw (C1v9) [e:main,color=myRed,->] to (C1v15);
			\draw (C1v5) [e:main,color=myRed,->] to (C1v11);
			\draw (C1v11) [e:main,color=myRed,->] to (C1v17);
			
			%%%%% %%%%% %%%%%
			
			%%%%% Center %%%%%
			
			\draw (C2v1und6) [e:main,->] to (C2v3und2);
			\draw (C2v3und2) [e:main,->] to (C2v5und4);
			\draw (C2v5und4) [e:main,->] to (C2v1und6);
			
			\draw (C2v7) [e:main,->] to (C2v8);
			\draw (C2v8) [e:main,->] to (C2v9);
			\draw (C2v9) [e:main,->] to (C2v10);
			\draw (C2v10) [e:main,->] to (C2v11);
			\draw (C2v11) [e:main,->] to (C2v12);
			\draw (C2v12) [e:main,->] to (C2v7);
			
			\draw (C2v13und18) [e:main,->,bend left] to (C2v15und14);
			\draw (C2v15und14) [e:main,->,bend left] to (C2v17und16);
			\draw (C2v17und16) [e:main,->,bend left] to (C2v13und18);
			
			\draw (C2v15und14) [e:main,color=myRed,<-] to (C2v8);
			\draw (C2v8) [e:main,color=myRed,<-] to (C2v3und2);
			\draw (C2v17und16) [e:main,color=myRed,<-] to (C2v10);
			\draw (C2v10) [e:main,color=myRed,<-] to (C2v5und4);
			\draw (C2v13und18) [e:main,color=myRed,<-] to (C2v12);
			\draw (C2v12) [e:main,color=myRed,<-] to (C2v1und6);
			
			\draw (C2v1und6) [e:main,color=myViolet,<-] to (C2v7);
			\draw (C2v7) [e:main,color=myViolet,<-] to (C2v13und18);
			\draw (C2v3und2) [e:main,color=myViolet,<-] to (C2v9);
			\draw (C2v9) [e:main,color=myViolet,<-] to (C2v15und14);
			\draw (C2v5und4) [e:main,color=myViolet,<-] to (C2v11);
			\draw (C2v11) [e:main,color=myViolet,<-] to (C2v17und16);
			
			%%%%% %%%%% %%%%%
			
			%%%%% Right Center %%%%%
			
			\draw (C3v1) [e:main] to (C3v2);
			\drawMatEdge{C3v2}{C3v3}
			\draw (C3v3) [e:main] to (C3v4);
			\drawMatEdge{C3v4}{C3v5}
			\draw (C3v5) [e:main] to (C3v6);
			\drawMatEdge{C3v6}{C3v1}
			
			\drawMatEdge{C3v7}{C3v8}
			\draw (C3v8) [e:main] to (C3v9);
			\drawMatEdge{C3v9}{C3v10}
			\draw (C3v10) [e:main] to (C3v11);
			\drawMatEdge{C3v11}{C3v12}
			\draw (C3v12) [e:main] to (C3v13);
			\drawMatEdge{C3v13}{C3v14}
			\draw (C3v14) [e:main] to (C3v15);
			\drawMatEdge{C3v15}{C3v16}
			\draw (C3v16) [e:main] to (C3v17);
			\drawMatEdge{C3v17}{C3v18}
			\draw (C3v18) [e:main] to (C3v7);
			
			\draw (C3v19) [e:main] to (C3v20);
			\drawMatEdge{C3v20}{C3v21}
			\draw (C3v21) [e:main] to (C3v22);
			\drawMatEdge{C3v22}{C3v23}
			\draw (C3v23) [e:main] to (C3v24);
			\drawMatEdge{C3v24}{C3v19}
			%		\draw (C3v19) [e:main,bend left] to (C3v20);
			%		\drawMatEdgeCurved{C3v20}{C3v21}
			%		\draw (C3v21) [e:main,bend left] to (C3v22);
			%		\drawMatEdgeCurved{C3v22}{C3v23}
			%		\draw (C3v23) [e:main,bend left] to (C3v24);
			%		\drawMatEdgeCurved{C3v24}{C3v19}
			
			\draw (C3v2) [e:main,color=myRed] to (C3v10);
			\draw (C3v9) [e:main,color=myRed] to (C3v20);
			\draw (C3v4) [e:main,color=myRed] to (C3v14);
			\draw (C3v13) [e:main,color=myRed] to (C3v22);
			\draw (C3v6) [e:main,color=myRed] to (C3v18);
			\draw (C3v17) [e:main,color=myRed] to (C3v24);
			
			\draw (C3v1) [e:main,color=myViolet] to (C3v7);
			\draw (C3v8) [e:main,color=myViolet] to (C3v19);
			\draw (C3v3) [e:main,color=myViolet] to (C3v11);
			\draw (C3v12) [e:main,color=myViolet] to (C3v21);
			\draw (C3v5) [e:main,color=myViolet] to (C3v15);
			\draw (C3v16) [e:main,color=myViolet] to (C3v23);
			
			%%%%% %%%%% %%%%%
			
			%%%%% %%%%% %%%%%
			
			\end{pgfonlayer}
			
			%%%%% %%%%% %%%%%
			
			%%%%% Background %%%%%
			\begin{pgfonlayer}{background}
			
			\draw (C1v1) [e:marker] to (C1v2);
			\draw (C1v3) [e:marker] to (C1v4);
			\draw (C1v5) [e:marker] to (C1v6);
			
			\draw (C1v13) [e:marker] to (C1v14);
			\draw (C1v15) [e:marker] to (C1v16);
			\draw (C1v17) [e:marker] to (C1v18);
			
			\end{pgfonlayer}	
			%%%%% %%%%% %%%%%
			
			%%%%% Foreground %%%%%
			\begin{pgfonlayer}{foreground}

			\end{pgfonlayer}
			%%%%% %%%%% %%%%%
			\end{tikzpicture}
		\end{center}
		\caption{The construction of the bipartite matching grid of order $3$ from the (directed) cylindrical grid of order $3$}
		\label{fig:matgrid}
	\end{figure}
	
	\begin{lemma}
		\label{thm:mat_grid_high_pmw}
		$\pmw{\MatGrid{n}}\ge \frac{1}{3}k$.
	\end{lemma}
	\begin{proof}
		Let $k\in\mathds{N}$ be a positive integer and $\CylGrid{k}$ the cylindrical grid of order $k$.
		By \cref{thm:cyl_grid_high_cw}, $\frac{1}{3}k\le\frac{1}{2}\cycWidth{\CylGrid{k}}$.
		Then the graph $\CylGridRetract{k}$ obtained from $\CylGrid{k}$ by contracting every second edge in the innermost and the outermost of its consecutive cycles as in the definition of the bipartite matching grid also has \cyclewidth at most $\frac{2}{3}k$ as we are only contracting already contractible edges and therefore do not create new directed cycles in our graph.
		By applying \cref{thm:pmw_cyw}, we obtain $\frac{1}{3}k\le\frac{1}{2}\cycWidth{\CylGridRetract{k}}\leq\pmw{\MatGrid{k}}$.
	\end{proof}
	
	With this last observation at hand, we can derive our main
	theorem. Assume that a graph $G$ has high  perfect matching
	width. By \cref{cor:cywgridthm} and \cref{thm:pmw_cyw} this implies high \cyclewidth for all $M$-directions of~$G$. This, in turn, implies large cylindrical
	grids as butterfly minors on those $M$-directions.
	Now \cref{lemma:mcguigmatminors} lets us translate these cylindrical grids into matching minors of $G$ and as the perfect matching width of $G$ is bounded from below by the width of its matching minors as we have observed in \cref{thm:pmw_minor_closed}, we obtain the grid theorem for bipartite matching covered graphs.
	
	\begin{theorem}
		There is a function $f\colon\N\rightarrow\N$ such that every matching
		covered bipartite graph $G$ either satisfies $\pmw{G}\leq f(k)$, or
		contains the bipartite matching grid of order $k$ as a matching minor.
	\end{theorem}

	\bibliographystyle{alphaurl}
	\bibliography{literature}

\end{document}